\documentclass[english,11pt]{amsart}
\usepackage{amsopn,amsthm,amsfonts,amsmath,amssymb,amscd}
\usepackage{babel}
\usepackage{nicematrix}
\usepackage{mathtools}
\usepackage{tikz}

\newtheorem{Theorem}{Theorem}[section]
\newtheorem{Lemma}[Theorem]{Lemma}

\newtheorem{Corollary}[Theorem]{Corollary}
\newtheorem{Example}[Theorem]{Example}

\newenvironment{Proof*}{{\it Proof.}}

\newcommand{\RR}{\mathbb{R}}
\newcommand{\ZZ}{\mathbb{Z}}

\newcommand{\Tr}{\mathrm{tr}}

\newcommand{\OO}{\mathcal{O}}

\newcommand{\Stab}{\mathrm{Stab}}

\begin{document}

\title{Products of idempotents in a quaternion ring}

\author{David Dol\v zan}

\address{D.~Dol\v zan:~Department of Mathematics, Faculty of Mathematics
and Physics, University of Ljubljana, Jadranska 19, SI-1000 Ljubljana, Slovenia, and Institute of Mathematics, Physics and Mechanics, Jadranska 19, SI-1000 Ljubljana, Slovenia; e-mail: 
david.dolzan@fmf.uni-lj.si}

\subjclass[2020]{16P10, 16U40, 16U99, 11R52} 
\keywords{finite ring, idempotent, quaternion, matrix}
\thanks{The author acknowledges the financial support from the Slovenian Research Agency  (research core funding No. P1-0222)}

\begin{abstract}
Let $R$ be a finite commutative local principal ring, and let $H(R)$ denote the corresponding quaternion ring. We show that an element of $H(R)$ is a product of idempotents if and only if it can be expressed as a product of two idempotents. Moreover, we obtain an explicit formula for the number of elements of $H(R)$ admitting such a factorization.
\end{abstract}

\maketitle 

 \section{Introduction}

\bigskip

Let $R$ be a commutative ring. The set
\begin{equation*}
H(R)= \{r_1+ r_2i + r_3j + r_4k : r_i \in R\}=R \oplus Ri \oplus Rj \oplus Rk,
\end{equation*}
together with the relations $ i^2 = j^2 = k^2 = ijk = -1$, and $ij = -ji$, 
forms a (noncommutative) ring, called the \emph{quaternion ring} over $R$. This construction generalizes Hamilton’s division ring of real quaternions $H(\RR)$.

In recent years, quaternion rings and their properties have attracted considerable attention.  Among others, the structure of the rings $H(\ZZ_p)$ and $H(\ZZ_n)$ was investigated in \cite{aris1,aris2} and \cite{mig2,mig1}, respectively. Further structural and functional aspects of quaternion rings were studied in \cite{ghara,ghara2}. When $2$ is invertible in $R$, the structure of $H(R)$ was described in \cite{cher22}, while systems of matrix equations over $H(R)$ were examined in \cite{xie}. More recently, representations of elements of quaternion rings as sums of exceptional units were considered in \cite{cherdol}.



The problem of decomposing elements of an algebra as products of idempotents dates back at least to 1966, when Howie \cite{howie} showed that every non-injective mapping on a finite set can be written as a product of idempotents. It is also well known that every singular $n$-by-$n$ matrix over a field is a product of idempotent matrices \cite{erdos}. This result was later extended to singular matrices over division rings and commutative Euclidean domains in \cite{laffey}, and to integer matrices in \cite{laffey1}. Further generalizations to noncommutative rings and semirings were obtained in \cite{alahmadi,alahmadi1}. A comprehensive survey of results on products of idempotents can be found in \cite{jain}.

More recently, related factorization problems have been studied for small matrix rings. In particular, Calug\u{a}reanu investigated products of two idempotents in the ring of $2$-by-$2$ matrices over a general domain in \cite{caluga}, and later examined which $2$-by-$2$ idempotent matrices can be written as products of two nilpotent matrices \cite{caluga1}. A characterization of singular $2$-by-$2$ matrices over commutative domains that are products of two idempotents or two nilpotents was obtained in \cite{caluga2}.

In this paper, we study products of idempotents in quaternion rings over finite commutative local principal rings. Our first main result shows that in this setting, the set of elements representable as a product of idempotents coincides with the set of products of two idempotents (Theorem~\ref{twoisall}). Using this characterization, we then determine the exact number of elements of a quaternion ring over a finite commutative local principal ring that admit such a factorization (Theorem~\ref{number}).

The paper is organized as follows. In Section~2, we collect the necessary definitions and preliminary results. Section~3 contains the main results: we characterize products of idempotents in quaternion rings and derive an explicit counting formula for such elements.

\bigskip

 \section{Definitions and preliminaries}
\bigskip

All rings in our paper will be finite rings with identity. 
For a ring $R$, $Z(R)$ will denote its centre and $I(R)$ will denote the set of all its idempotents. We say that the idempotents $0,1 \in R$ are trivial idempotents.
The group of units in $R$ will be denoted by $U(R)$ and the Jacobson radical of $R$ by $J(R)$.

We will denote the $2$-by-$2$ matrix ring with entries in $R$ by $M_2(R)$, while the group of invertible matrices therein will be denoted by $GL_2(R)$.


For any $a, b \in R$, let $M(a,b)=\left(\begin{array}{cc}
a & b\\
0 & 0
\end{array}\right) \in M_2(R)$. 

The group $GL_2(R)$ acts on $M_2(R)$ by conjugation. We will denote the orbit of an element $X \in M_2(R)$ for this action by $\OO_X$ and its stabilizer by $\Stab(X)$.

We shall need the following lemma.

\begin{Lemma}
\label{idemchar}
  Let $R$ be a finite commutative local ring.
   If $A \in M_2(R)$ is a nontrivial idempotent, then $A \in \OO_{M(1,0)}$. 
\end{Lemma}
\begin{proof}
By \cite[Theorem 2.1]{dolzan} there exists an invertible matrix $P \in M_2(R)$ and a $0/1$ diagonal matrix $B \in  M_2(R)$
such that $A = PBP^{-1}$. But since $B$ is a nontrivial idempotent, we only have two possibilities: either $B=M(1,0)$ or $B=\left(\begin{array}{cc}
0 & 0\\
0 & 1
\end{array}\right)$. However, we have $\left(\begin{array}{cc}
0 & 0\\
0 & 1
\end{array}\right) \in \OO_{M(1,0)}$, so the lemma is proven.   
\end{proof}

It follows from \cite[Theorem 2]{ragha} that every finite local ring has cardinality $p^{nr}$ for some prime number $p$ and some integers $n, r$. Furthermore, the Jacobson radical $J(R)$ is of cardinality $p^{(n-1)r}$ and the factor ring $R/J(R)$ is a field with $p^r$ elements.

We also have the following lemma.

\begin{Lemma}
\label{pid}    
Let $R$ be a finite local principal ring of cardinality $q^{n}$, where $q=p^r$ for some prime number $p$ and integers $n, r$ with $R/J(R) \simeq GF(q)$. Then there exists $x \in J(R)$ such that $J(R)^k=(x^k)$ for every $k \in \{0,1,\ldots,n\}$. In particular, $|J(R)^k|=q^{n-k}$ for every $k \in \{0,1,\ldots,n\}$.
\end{Lemma}
\begin{proof}
    Since $R$ is a principal ring, we have $J(R)=(x)$ for some $x \in R$, so $J(R)^k=(x^k)$ for every $k \in \{0,1,\ldots,n\}$. Denote by $T=GR(p^l,r)$ the Galois ring of characteristic $p^l$, and suppose $\beta$ is an integer such that $J(R)^{\beta} = 0$ and $J(R)^{\beta-1} \neq 0$. By \cite[Lemma XVII.4]{mcdonald}, there exist positive integers $s,t$ such that $\beta=(l-1)s+t$ and $R$ is isomorphic to $T \oplus Tx \oplus \ldots \oplus Tx^{s-1}$ as a $T$-module. Furthermore, we have $T$-module isomorphisms $Tx^i \simeq T$ for $i \leq t-1$, and $Tx^i \simeq Tp$ for $t \leq i \leq s$. 
    Since $|T|=p^{lr}$ and $|Tp|=p^{(l-1)r}$, we get $p^{nr}=|R|=p^{r(lt+(l-1)(s-t))}=p^{r\beta}$, so $\beta=n$. Since for every $k=0,1,\ldots, n-1$, $J(R)^k/J(R)^{k+1}$ is a $R/J(R)$ vector space, we also have $|J(R)^k|=p^{(n-k)r}$ for every $k \in \{0,1,\ldots,n\}$.
\end{proof}

The following lemma will come in handy throughout the paper. The proof is straightforward.

\begin{Lemma}
\label{obvious}
   Let $R$ be a finite ring. Then $U(R)+J(R) \subseteq U(R)$. 
\end{Lemma}

We also have the following theorem that can be deduced from \cite[Theorem 3.10]{cher}, but we include the proof here for the sake of completeness.

\begin{Theorem}
\label{izo}
Let $R$ be a finite commutative local ring of order $p^{nr}$ such that $R/J(R)$ is a field with $p^r$ elements, for some odd prime number $p$ and integers $n, r$. Then $H(R) \simeq M_2(R)$.
\end{Theorem}
\begin{proof}
By the proof of \cite[Theorem 3.10]{cher}, we only need to prove that there exist $\alpha, \beta \in R$ such that $1+\alpha^2+\beta^2=0$.
We prove this by induction on $n$. If $n = 1$, then $R$ is a finite field, so the statement follows (see for example \cite[Theorem 6.27]{lidl}). So, assume $n > 1$.

There exists an integer $k$ such that $J(R)^{k+1}=0$ and $J(R)^{k} \neq 0$. Define $S=R/J(R)^{k}$. By the induction hypothesis there exist $\overline{\alpha}, \overline{\beta} \in S$ such that $\overline{1} + \overline{\alpha}^2 + \overline{\beta}^2 = 0$ in $S$. Therefore there exist $\alpha, \beta \in R$ and $w \in J(R)^k$ such that $1+\alpha^2+\beta^2 = w$. This implies that at least one of $\alpha, \beta$ is not in $J(R)$. We can assume without any loss of generality that $\alpha \notin J(R)$, so $\alpha \in U(R)$.
Since $2 \in U(R)$, we can define $\alpha'=\alpha - w(2\alpha)^{-1}$ and observe that $\alpha'^2=\alpha^2-w=-1-\beta^2$, so $1+\alpha'^2+\beta^2=0$.
\end{proof}

\bigskip

 \section{Products of idempotents}
\bigskip

In this section we turn to the product of idempotents in a quaternion ring over a finite commutative local principal ring. The restriction to finite commutative local principal rings is dictated by the methods used here. Locality forces idempotents in $M_2(R)$ to be conjugate to standard diagonal ones, which makes products of idempotents amenable to orbit considerations. The principal ideal assumption gives a uniform description $J(R)=(x^k)$, allowing us to parametrize conjugacy classes $\OO_{M(a,b)}$
effectively and to evaluate their sizes explicitly, which is essential for the counting formula in Theorem~\ref{number}. Without principality (or without locality), the structure of the powers of the Jacobson radical and the resulting orbit stratification can be substantially more complicated.

We begin by determining the number of idempotent elements.

\begin{Lemma}
\label{noifid}
Let $R$ be a finite commutative local ring of
cardinality $p^{nr}$ such that $R/J(R)$ is a field with $q=p^r$ elements.
Then $|I(H(R))|=2+q^{3n-2}(q^2-1)$ if $p > 2$ and $|I(H(R))|=2$ otherwise.
\end{Lemma}
\begin{proof}
  Suppose firstly that $p > 2$. By Theorem \ref{izo}, we know that $H(R) \simeq M_2(R)$. Lemmas 2.2 and 2.3 from \cite{dolzan} now give us $|I(H(R))|=2+q^{3n-2}(q^2-1)$.
  On the other hand, if $p=2$, $2$ is a zero-divisor in $R$, so $H(R)$ is a finite local ring by \cite[Lemma 3.1]{cherdol} and as such contains no nontrivial idempotents.
\end{proof}

We shall need the following technical lemma.

\begin{Lemma}
\label{tech}    
Let $R$ be a finite commutative local principal ring of cardinality $q^{n}$, where $q=p^r$ for some prime number $p$ and integers $n, r$ with $R/J(R) \simeq GF(q)$. Choose $a, a', b, b' \in R$.
Then $\OO_{M(a,b)}=\OO_{M(a',b')}$ if and only if $a=a'$ and there exists $u \in U(R)$ such that $b'-ub \in (a)$.
\end{Lemma}
\begin{proof}
$(\Rightarrow)$:
Suppose $\OO_{M(a',b')} = \OO_{M(a,b)}$, so there exists an invertible matrix $P=\left(\begin{array}{cc}
\alpha & \beta\\
\gamma & \delta
\end{array}\right) \in M_2(R)$ such that $PM(a,b)=M(a',b')P$. Since $\Tr(AB)=\Tr(BA)$ for any matrices $A,B$, we have $\Tr(M(a,b))=\Tr(M(a',b'))$, so $a=a'$.
By Lemma \ref{pid}, there exists $x \in J(R)$ such that $J(R)^k=(x^k)$ for any $k=0,1,\ldots,n$. Therefore there exists $0 \leq l \leq n$ such that $a=u_ax^l$ for some $u_a \in U(R)$. 
Suppose that $b=u_bx^t$ and $b'=u_{b'}x^s$ for some integers $t,s$ and some invertible $u_b,u_{b'} \in R$. We can assume without loss of generality that $t \leq s$, otherwise we can swap $b$ and $b'$, since we have that $b'-ub \in (a)$ for some $u \in U(R)$ if and only if $b-vb' \in (a)$ for some $v \in U(R)$.

Suppose now that $b-ub' \notin (x^l)$ for every $u \in U(R)$. Then $t < l$, which also implies that $b \neq 0$.
Now, the equation $PM(a,b)=M(a',b')P$ gives us $\gamma b=0$, so $\gamma \in J(R)$. Since $P$ is an invertible matrix, this implies that $\alpha, \delta \in U(R)$. But $PM(a,b)=M(a',b')P$ also implies that $\alpha b = a \beta + b' \delta$, therefore
$b' - \delta^{-1}\alpha b \in (a)$.

$(\Leftarrow)$:
If $a'=a$ and $b'=ub+ta$ for some $u \in U(R)$ and some $t \in R$, then we have
  $\left(\begin{array}{cc}
1 & t\\
0 & 1
\end{array}\right)^{-1}\left(\begin{array}{cc}
1 & 0\\
0 & u
\end{array}\right)^{-1}M(a,b)\left(\begin{array}{cc}
1 & 0\\
0 & u
\end{array}\right)\left(\begin{array}{cc}
1 & t\\
0 & 1
\end{array}\right)=M(a',b')$.
\end{proof}

The next lemma will be crucial in our investigation.

\begin{Lemma}
\label{productof2idempotents}    
Let $R$ be a finite commutative local principal ring. Then $A \in M_2(R)$ is a product of two idempotent matrices if and only if $A=I$ or there exist $a, b \in R$ such that $A \in \OO_{M(a,b)}$.
\end{Lemma}
\begin{proof}
$(\Rightarrow)$: Suppose $A=E_1E_2$ for some idempotents $E_1, E_2 \in M_2(R)$.
If $E_1 \in Z(M_2(R))$, then $A$ is an idempotent. 
By Lemma \ref{idemchar}, $A=I$, $A \in \OO_{M(0,0)}$, or $A \in \OO_{M(1,0)}$.

Suppose now that $E_1$ is non-central.
Again, using Lemma \ref{idemchar}, we have 
$PE_1P^{-1}=M(1,0)$ for some invertible matrix $P$.
Thus $PAP^{-1}= M(1,0)F$ for some idempotent matrix $F=\left(\begin{array}{cc}
a & b\\
c & d
\end{array}\right) \in M_2(R)$, so $PAP^{-1}= M(a,b)$, therefore $A \in \OO_{M(a,b)}$.

$(\Leftarrow)$:
Let $|R|=q^n$ with $R/J(R) \simeq GF(q)$.
Since conjugation preserves idempotents, we only have to prove that $M(a,b) \in M_2(R)$ is a product of two idempotents for any $a, b \in R$.

Let $l, k$ be integers and $a \in J(R)^l \setminus J(R)^{l+1}, b \in J(R)^k \setminus J(R)^{k+1}$.
By Lemma \ref{pid}, there exists $x \in J(R)$ such that $J(R)^t=(x^t)$ for every $t=0,1,\ldots, n$. Therefore $b=ux^k$ for some $u \in U(R)$ and $a=vx^l$ for some $v \in U(R)$. Suppose firstly that $l \geq k$. Then observe that $F=\left(\begin{array}{cc}
a & b\\
u^{-1}v(1-a)x^{l-k} & 1-a
\end{array}\right)$ is an idempotent matrix and $M(a,b)=\left(\begin{array}{cc}
1 & 0\\
0 & 0
\end{array}\right)F$. Note that in particular, this implies that $M(a,a)$ is a product of two idempotents.

On the other hand, if $l < k$, then by Lemma \ref{tech} we have that $M(a,b) \in \OO_{M(a,b')}$ for $b'=wb+ta$ with any $w \in U(R)$ and any $t \in R$. Thus we can choose $t=1-wuv^{-1}x^{k-l}$, so $b'=a$. Since by the above, any element in $\OO_{M(a,a)}$ is a product of two idempotents, we have $M(a,b) \in \OO_{M(a,a)}$ also being a product of two idempotents. 
%
%
\end{proof}

We now immediately have the following theorem, which is the first main result of this paper.

\begin{Theorem}
\label{twoisall}
   Let $R$ be a finite commutative local principal ring. Then $x \in H(R)$ is a product of  $r \geq 1$ idempotents if and only if $x$ is a product of two idempotents.
\end{Theorem}
\begin{proof}
   One implication is obvious. 
   
   So, let us assume that $x \in H(R)$ is a product of $r \geq 1$ idempotents. Obviously, if $r=1$ then $x=x \cdot 1$ is also a product of two idempotents. Therefore, we shall henceforth assume that $r \geq 3$.
   
   Now, if $2 \in R$ is not invertible, then $H(R)$ is a local ring by \cite[Lemma 3.1]{cherdol}, so it only contains trivial idempotents and thus $x=0$ or $x=1$, which are both products of two idempotents.
   
   We may therefore assume that $2 \in R$ is invertible. Then we have $H(R) \simeq M_2(R)$ by Theorem \ref{izo}, so we have to prove that any 
   $A \in M_2(R)$ that is a product of $r$ idempotent matrices for some $r \geq 3$, can also be written as a product of two idempotents. Therefore, assume that $A=E_1E_2\ldots E_r$ for idempotent matrices $E_1, E_2, \ldots, E_r$. We can also assume without any loss of generality that $E_1$ is a nontrivial idempotent. By Lemma \ref{idemchar}, there exists an invertible matrix $P \in M_2(R)$ such that $PE_1P^{-1}=M(1,0)$. Thus
   $PAP^{-1}=M(1,0)F_2\ldots F_r$ for some (idempotent) matrices $F_2,\ldots,F_r$. This implies that $PAP^{-1}=M(a,b)$ for some $a,b \in R$. By Lemma \ref{productof2idempotents}, this implies that $A$ is a product of two idempotents.
\end{proof}

We shall now use this theorem to investigate how many elements of $H(R)$ we can decompose as a product of idempotents. In order to calculate this, we shall need the following two lemmas. The first lemma concerns the sizes of orbits.

\begin{Lemma}
\label{orbit}    
 Let $R$ be a finite commutative local principal ring of cardinality $q^{n}$, where $q=p^r$ for some prime number $p$ and integers $n, r$ with $R/J(R) \simeq GF(q)$. Choose $a, b \in R$ and let $0 \leq k,l \leq n$ be integers such that $a \in J(R)^l \setminus J(R)^{l+1}$ and $b \in J(R)^k \setminus J(R)^{k+1}$. Then
 $$\left|\OO_{M(a,b)}\right|=
 \begin{cases}
 q^{2n-k-1}(q^{2}-1); \text { if } k  < l \leq n, \\
 q^{2n-2l-1}(q+1); \text { if } a \neq 0, b=0, \\
 1; \text { if } a=b=0.
 \end{cases}$$ 
 \end{Lemma}
\begin{proof}
Obviously, we have $\left|\OO_{M(0,0)}\right|=1$. 

We know that $\left|\OO_{M(a,b)}\right|=\frac{|GL_2(R)|}{|\Stab(M(a,b))|}$ and that $|GL_2(R)|=q^{4n-3}(q-1)(q^{2}-1)$ (see for example \cite[Lemma 2.3]{dolzan}), so we have to calculate the sizes of the respective stabilizers.


Let us calculate $|\Stab(M(a,0))|$ for $a \neq 0$. Observe that $P=\left(\begin{array}{cc}
\alpha & \beta\\
\gamma & \delta
\end{array}\right) \in \Stab(M(a,0))$ if and only if $\gamma a = \beta a = 0$, so $\beta, \gamma \in J(R)^{n-l}$. Therefore, we have $q^{2l}$ possibilities for choosing $\beta$ and $\gamma$.
Since $a \neq 0$, we have $\beta, \gamma \in J(R)$, so $\alpha, \delta \in U(R)$.
This gives us
$|\Stab(M(a,0))|=q^{2(l+n-1)}(q-1)^2$ and therefore
$\left|\OO_{M(a,0)}\right|=q^{2n-2l-1}(q+1)$.

Let us finally calculate $|\Stab(M(a,b))|$ in the case $k < l \leq n$.
   Observe that $P=\left(\begin{array}{cc}
\alpha & \beta\\
\gamma & \delta
\end{array}\right) \in \Stab(M(a,b))$ if and only if $PM(a,b)=M(a,b)P$ if and only if $\gamma a = \gamma b = 0$ and $a\beta=(\alpha - \delta)b$.
Since $k < l$, the conditions $\gamma a = \gamma b = 0$ are equivalent to the fact that $\gamma \in J(R)^{n-k}$. Since $|J(R)^{n-k}|=q^{k}$, we have $q^{k}$ possibilities for choosing $\gamma$.
By Lemma \ref{pid}, there also exists $x \in J(R)$ such that $J(R)^t=(x^t)$ for every $t=0,1,\ldots, n$. Therefore $a=ux^l$ and $b=vx^k$ for some $u, v \in U(R)$. 
Since $\gamma \in J(R)$, the fact that $P$ is an invertible matrix implies that $\alpha, \delta \in U(R)$. 
Choose any $\beta \in R$ and any $\alpha \in U(R)$, giving us $q^n$ and $q^{n-1}(q-1)$ possibilities, respectively. The condition $a\beta=(\alpha - \delta)b$ can be rewritten as $v^{-1}u\beta x^l=(\alpha-\delta)x^k$. Denote $v^{-1}u\beta=\beta_0+\beta_1x+\ldots+\beta_{n-1}x^{n-1}$ and $\alpha - \delta=\epsilon_0+\epsilon_1x+\ldots+\epsilon_{n-1}x^{n-1}$. This implies $\epsilon_0=\epsilon_1=\ldots=\epsilon_{l-k-1}=0$ and $\epsilon_{l-k}=\beta_0, \epsilon_{l-k+1}=\beta_1,\ldots,\epsilon_{n-k-1}=\beta_{n-l}$. (Note here that $k<l$ implies that we always have $\epsilon_0=0$, so $\delta$ is automatically invertible if $\alpha$ is invertible.) Therefore, the $k$ elements $\epsilon_{n-k}, \epsilon_{n-k+1}, \ldots, \epsilon_{n-1}$ can be chosen arbitrarily (giving us $q$ choices for each one, since $J(R)^m/J(R)^{m+1}$ is a one-dimensional vector space over $R/J(R)$ by Lemma \ref{pid}).
Therefore, we have
$|\Stab(M(a,b))|=q^{k}q^{n}q^{n-1}(q-1)q^{k}=q^{2k+2n-1}(q-1)$ and thus
$\left|\OO_{M(a,b)}\right|=q^{2(n-k-1)}(q^{2}-1)$.
\end{proof}

The next lemma will help us determine when the specific orbits coincide.

\begin{Lemma}
\label{whichorbit}    
 Let $R$ be a finite commutative local principal ring of cardinality $q^{n}$, where $q=p^r$ for some prime number $p$ and integers $n, r$ with $R/J(R) \simeq GF(q)$. Choose integers $0 \leq k, l \leq n$ and choose $a \in J(R)^l \setminus J(R)^{l+1}$ and $b \in J(R)^k \setminus J(R)^{k+1}$. Then the following statements hold.
 \begin{enumerate}
 \item 
 If $k \geq l$, then $\{ta+wb; t \in R, w \in U(R)\}=RJ(R)^{l}$.
 
 \item
 If $k < l$, then $\{ta+wb; t \in R, w \in U(R)\}=U(R)J(R)^{k}$.
 \end{enumerate}
 \end{Lemma}
\begin{proof}
 By Lemma \ref{pid}, there exists $x \in J(R)$ such that $J(R)^m=(x^m)$ for every $m=0,1,\ldots, n$. Therefore $a=ux^l$ for some $u \in U(R)$ and $b=vx^k$ for some $v \in U(R)$.
   \begin{enumerate}
  \item    
   Since $k \geq l$, we obviously have $ta+wb \subseteq RJ(R)^{l}$ for every $t \in R$ and every $w \in U(R)$.
   
   \noindent So, choose any $y \in R$. Write $y=y_0+y_1x+\ldots +y_{n-1}x^{n-1}$ for some elements $y_0, y_1, \ldots, y_{n-1} \in R$. Since $J(R)^m/J(R)^{m+1}$ is a $R/J(R)$-vector space by Lemma \ref{pid}, we can actually choose $y_0, y_1, \ldots, y_{n-1} \in U(R) \cup \{0\}$. 
      If $y_{k-l}=0$, denote $y'_{k-l}=1$, otherwise denote $y'_{k-l}=y_{k-l}$. Then we can choose $t=u^{-1}(y - y'_{k-l}x^{k-l}) \in R$ and $w=v^{-1}y'_{k-l} \in U(R)$ and observe that $ta+wb=yx^l$.
   Thus, $RJ(R)^{l} = \{ta+wb; t \in R, w \in U(R)\}$.

   \item
   Since $k < l$, we have $ta+wb=(tux^{l-k}+wv)x^k \subseteq (J(R)+U(R))J(R)^{k}$ for every $t \in R$ and every $w \in U(R)$. By Lemma \ref{obvious}, $U(R)+J(R) \subseteq U(R)$, so $ta+wb \in U(R)J(R)^k$.

   \noindent 
   On the other hand, choose any $y \in U(R)$ and write $y=y_0+y_1x+\ldots +y_{n-1}x^{n-1}$ for some $y_0, y_1, \ldots, y_{n-1} \in R$. We can now choose $t=u^{-1}y_{l-k} \in R$ and $w=v^{-1}(y - y_{l-k}x^{l-k}) \in U(R)$ and observe that $ta+wb=yx^k$.
   Thus, $U(R)J(R)^{k} = \{ta+wb; t \in R, w \in U(R)\}$.
 \end{enumerate}
 With this, the lemma is proven.
\end{proof}

This lemma now immediately gives us the following corollary.

\begin{Corollary}
\label{coro}
    Let $R$ be a finite commutative local principal ring of cardinality $q^{n}$, where $q=p^r$ for some prime number $p$ and integers $n, r$ with $R/J(R) \simeq GF(q)$. Choose integers $0 \leq k, l \leq n$ and choose $a \in J(R)^l \setminus J(R)^{l+1}$ and $b \in J(R)^k \setminus J(R)^{k+1}$. Let $J(R)=(x)$ for some $x \in R$. Then the following statements hold.
 \begin{enumerate}
 \item 
 If $k \geq l$, then $\OO_{M(a,b)}=\OO_{M(a,0)}$.
 
 \item
 If $k < l$, then $\OO_{M(a,b)}=\OO_{M(a,x^k)}$.
 \end{enumerate}
\end{Corollary}
\begin{proof}
  This follows directly from Lemma \ref{tech} and Lemma \ref{whichorbit}.  
\end{proof}

We now have our second main result.

\begin{Theorem}
\label{number}
  Let $R$ be a finite commutative local principal ring of cardinality $q^{n}$, where $q=p^r$ for some prime number $p$ and integers $n, r$ with $R/J(R) \simeq GF(q)$. Then the number of elements in $H(R)$ that can be decomposed as a product of idempotents is equal to $2$, if $2 \in R$ is not invertible, and $$\frac{q^2+1+q^{3n}(q+1)^2}{q^2+q+1}$$ otherwise.  
\end{Theorem}
\begin{proof}
  Denote $Z=\{z \in H(R); z$ is a product of idempotents$\}$.
  If $2 \in R$ is not invertible, then $H(R)$ is a local ring by \cite[Lemma 3.1]{cherdol}, so it only contains trivial idempotents and thus $Z=\{0,1\}$.
   
   We can therefore assume that $2 \in R$ is invertible. Then we have $H(R) \simeq M_2(R)$ by Theorem \ref{izo}, so 
   Lemma \ref{productof2idempotents} and Theorem \ref{twoisall} tell us that $|Z|=|\bigcup_{a,b \in R} \OO_{M(a,b)}|+1$.   
   Again using Lemma \ref{pid}, we know there exists $x \in J(R)$ such that $J(R)^t=(x^t)$ for every $t=0,1,\ldots, n$.
   Suppose $a \in J(R)^l \setminus J(R)^{l+1}$ and $b \in J(R)^k \setminus J(R)^{k+1}$ for some integer $0 \leq l \leq n-1$ and some integer $0 \leq k \leq n$.
   Using Corollary \ref{coro}, we therefore have 
   $\OO_{M(a,b)}=\OO_{M(a,0)}$ if $k \geq l$ and $\OO_{M(a,b)}=\OO_{M(a,x^k)}$ if $k < l$. This gives us   
   $$\left|\bigcup_{b \in R} \OO_{M(a,b)}\right|=\left|\OO_{M(a,0)}\right|+\sum_{k=0}^{l-1} \left|\OO_{M(a,x^k)}\right|.$$
   Now, we use Lemma \ref{orbit} to establish
   $$\left|\bigcup_{b \in R} \OO_{M(a,b)}\right|=q^{2n-2l-1}(q+1)+\sum_{k=0}^{l-1}q^{2(n-k-1)}(q^{2}-1), \text { if } a\neq 0, \text { and }$$
   $$\left|\bigcup_{b \in R} \OO_{M(0,b)}\right|=1+\sum_{k=0}^{n-1}q^{2(n-k-1)}(q^{2}-1).$$
   So, using Lemma \ref{tech}, after summing over all elements $a \in R$, we have 
   \begin{multline*}
     |Z|=2+\sum_{k=0}^{n-1}q^{2(n-k-1)}(q^{2}-1)+\\ \sum_{l=0}^{n-1}\left(q^{n-1-l}(q-1)\left(q^{2n-2l-1}(q+1)+\sum_{k=0}^{l-1}q^{2(n-k-1)}(q^{2}-1)\right)\right).
   \end{multline*}
   After evaluating the resulting geometric series, we obtain
    \begin{multline*}
    |Z|=2+(q^{2n}-1)+\frac{q(q^{3n}-1)(q^2-1)}{q^3-1}+\sum_{l=0}^{n-1}(q-1)q^{3(n-l)-1}(q^{2l}-1)= \\
   1+q^{2n}+\frac{q(q^{3n}-1)(q^2-1)}{q^3-1}+q^{3n}-q^{2n}-\frac{q^2(q-1)(q^{3n}-1)}{q^3-1},
    \end{multline*} and finally
    $$|Z|=\frac{q^2+1+q^{3n}(q+1)^2}{q^2+q+1}.$$
    \end{proof}

Let us illustrate the theorem in the following elementary example.

\begin{Example}
  Let $n$ be an integer. Theorem \ref{number} gives us that only $2$ elements in $H(\ZZ_{2^n})$ can be decomposed as products of idempotents, while for $\alpha=3^n$, we have $$\frac{10+16\alpha^3}{13}$$ elements in $H(\ZZ_{\alpha})$ that can be decomposed as products of idempotents. In particular, $898$ out of $2673$ noninvertible elements in $H(\ZZ_{9})$ can be decomposed as products of idempotents.
\end{Example}

\bigskip

\bigskip

{\bf Statements and Declarations} \\

The author states that there are no competing interests. 

\bigskip

\bibliographystyle{amsplain}
\bibliography{biblio}

\bigskip

\end{document}